\documentclass[a4paper, reqno]{amsart}

\usepackage{amsmath}
\usepackage{amsfonts}
\usepackage{amsthm}
\usepackage{graphicx}
\usepackage{eucal}
\usepackage{amscd}
\usepackage[all,2cell]{xy}
\usepackage{amssymb}
\usepackage{mathrsfs}
\usepackage{tikz-cd}
 \usepackage{tikz}
 \usetikzlibrary{positioning}
 \usepackage{tikz-cd}
 \tikzset{mynode/.style={draw,circle,inner sep=1pt,outer sep=0pt}}

\usepackage[colorlinks=true, linkcolor=blue, citecolor=blue, urlcolor=blue]{hyperref}

\newtheorem{teo}{Theorem}[section]
\newtheorem{cor}[teo]{Corollary}
\newtheorem{lem}[teo]{Lemma}
\newtheorem{defi}[teo]{Definition}

\newtheorem{prop}[teo]{Proposition}
\newtheorem{remark}[teo]{Remark}

\theoremstyle{definition}                
\newtheorem{example}[teo]{Example}
\newtheorem{examples}[teo]{Examples}

\newcommand{\X}{\ensuremath{\mathbb{X}}}
\newcommand{\A}{\ensuremath{\mathbb{A}}}

\newdir{ |>}{{}*!/-3.5pt/@{|}*!/-8pt/:(1,-.2)@^{>}*!/-8pt/:(1,+.2)@_{>}}

\dedicatory{}

\begin{document}

\title{QUASI-POINTED EXACTNESS AND TORSION}

\author{Marino Gran}
\address[Marino Gran]{Institut de Recherche en Math\'ematique et Physique, Universit\'e catholique de Louvain, Chemin du Cyclotron 2, 1348 Louvain-la-Neuve, Belgique}
\thanks{}
\email{marino.gran@uclouvain.be}

\author{George Janelidze}
\address[George Janelidze]{Department of Mathematics and Applied Mathematics, University of Cape Town, Rondebosch 7700, South Africa}
\thanks{}
\email{george.janelidze@uct.ac.za}

\keywords{Quasi-pointed categories, non-pointed torsion theory, pretorsion theory, near-torsion theory, reflective subcategory, coreflective subcategory, Lie-Rinehart algebras, internal groupoids, crossed modules.}

\subjclass[2020]{18E40, 18A30, 18A40, 18D40. }

\begin{abstract}
We show that, under some mild assumptions,  every quasi-pointed category $\mathcal{C}$ carries a canonical near-torsion theory $(\mathcal{T},\mathcal{F})$, where $\mathcal{T}$ consists of the $0$-supported objects, that is, those admitting a morphism to the initial object $0$, and $\mathcal{F}$ consists of the subterminal objects. The coreflection and the reflection of an object $A$ are then given by the projection $A\times 0\to A$ and by the strong-epimorphic part of $A\to 1$, respectively. Using the general construction of the `largest' non-pointed torsion theory associated with a near-torsion theory, we prove that $(\mathcal{T},\mathcal{F})$ is a (non-pointed) torsion theory in the full subcategory of those objects $A$ for which the canonical sequence $t(A)\to A\to f(A)$ is short exact, in the sense that it is both a kernel and a cokernel diagram. This subcategory is the whole of $\mathcal{C}$ whenever $\mathcal{C}$ is sequentiable, that is,  a regular protomodular and quasi-pointed category. We describe the resulting torsion theories in the category of Lie--Rinehart algebras and in the categories of crossed modules over a fixed group. We also show that $(\mathcal{T},\mathcal{F})$ is a torsion theory in the (non-protomodular) category of monoids over a group $G$ all of whose elements have finite order.
\end{abstract}

\date{\today}

\maketitle

\section{Introduction}

A category $\mathcal{C}$ with an initial object $0$ and a terminal object $1$ is \emph{quasi-pointed} when the unique morphism $0\to 1$ is a monomorphism. This notion was introduced by D. Bourn in \cite{[B2001]}, where he showed that one can prove many of the classical homological lemmas, such as the $3 \times 3$-Lemma and the Snake Lemma, in any \emph{sequentiable} category, that is, in any quasi-pointed regular protomodular category. 

The purpose of this paper is to show that every quasi-pointed category carries, under very mild assumptions, a canonical ``torsion theory'', in which the roles of the torsion and of the torsion-free objects are played by two natural classes of objects determined by $0$ and $1$. An object $A$ is \emph{$0$-supported} when it admits a (necessarily unique) morphism $A\to 0$, and \emph{subterminal} when $A\to 1$ is a monomorphism. Writing $\mathcal{T}$ and $\mathcal{F}$ for the full subcategories of $0$-supported and of subterminal objects, respectively, we prove that $\mathcal{T}$ is coreflective in $\mathcal{C}$ as soon as the product $A\times 0$ exists for any object $A$, the $A$-component of the coreflection $\varepsilon_A\colon t(A)\to A$ being the projection $A\times 0\to A$ (Corollary \ref{coreflection}). On the other hand, the subcategory of subterminal objects $\mathcal{F}$ is reflective in $\mathcal{C}$ as soon as every morphism $A\to 1$ admits a (strong epi, mono)-factorization, the $A$ component of the reflection $\eta_A\colon A\to f(A)$ being the strong-epimorphic part of this factorization (Corollary \ref{reflection}). Under these assumptions $(\mathcal{T},\mathcal{F})$ is then a \emph{near-torsion theory} in the sense of our previous paper \cite{[GJ2026]} (Theorem \ref{counit-projection}).

This is exactly the situation studied in \cite{[GJ2026]}, whose main motivation was to introduce a new notion of (non-pointed) torsion theory which does not use non-pointed exactness. The main example considered in that paper also suggested a natural notion of exact sequence of morphisms.
 There we introduced near-torsion theories, that is, pairs $(\mathcal{T},\mathcal{F})$ of a coreflective and a reflective subcategory satisfying two natural compatibility conditions, and we showed how to associate with any such pair the `largest' full subcategory $\mathcal{C}_{\mathcal{T},\mathcal{F}}$ of $\mathcal{C}$ on which $(\mathcal{T},\mathcal{F})$ restricts to a (non-pointed) torsion theory together with the corresponding subcategories $\mathring{\mathcal{T}}$ and $\mathring{\mathcal{F}}$ of $\mathcal{T}$ and $\mathcal{F}$, respectively. The present paper can thus be seen as a detailed study of an important instance of that general construction: the near-torsion theory of a quasi-pointed category. In this situation the general theory of \cite{[GJ2026]} simplifies considerably. We show that $\mathring{\mathcal{T}}=\mathcal{T}$, $\mathring{\mathcal{F}}=\mathcal{F}$, and that an object $A$ belongs to $\mathcal{C}_{\mathcal{T},\mathcal{F}}$ if and only if the canonical sequence
\begin{equation}\label{intro-SES}
\xymatrix{t(A)\ar[r]^-{\varepsilon_A}&A \ar[r]^-{\eta_A}&f(A)}
\end{equation}
is short exact in the sense of \cite{[B2001]}, that is, it is both a kernel diagram and a cokernel diagram. It turns out that the kernel condition holds automatically, hence the only requirement is that the square 
\begin{equation}\label{pushout}
\xymatrix{t(A)\ar[r] \ar[d]_-{\varepsilon_A} &0\ar[d] \\
A \ar[r]_-{\eta_A} & f(A)}
\end{equation}
 is a pushout (Theorem \ref{thm:3.2}). In particular, by Bourn's results on sequentiable categories \cite{[B2001]}, we get that $\mathcal{C}_{\mathcal{T},\mathcal{F}}=\mathcal{C}$ whenever $\mathcal{C}$ is sequentiable, so that $(\mathcal{T},\mathcal{F})$ is then a (non-pointed) torsion theory in $\mathcal{C}$. In the quasi-pointed context, $\mathcal{T}\cap\mathcal{F} =\mathcal Z$ is the full subcategory of initial objects, so that $(\mathcal{T},\mathcal{F})$ is also a \emph{pretorsion theory} in $\mathcal{C}_{\mathcal{T},\mathcal{F}}$ in the sense of \cite{[FF2020],[FFG2021]}, the exactness of \eqref{intro-SES} being then understood relatively to the ideal of morphisms factoring through an initial object in $\mathcal Z$.

We then describe this canonical torsion theory in the category of Lie--Rinehart algebras \cite{Rinehart}, and in the categories of internal reflexive graphs and groupoids in the category of groups over a fixed group $B$ (Examples \ref{Lie+crossed}).
In the case of the Lie--Rinehart algebras over $(R,A)$, where $R$ is a commutative unital ring and $A$ a commutative $R$-algebra, the torsion objects are the Lie--Rinehart algebras with trivial anchor and the torsion-free ones are the Lie--Rinehart subalgebras of $\mathsf{Der}(A)$. In the category of internal reflexive graphs over a fixed group $B$, the torsion objects are the ``points'' over $B$ \cite{[B1991]} and the torsion-free ones are the equivalence relations on $B$. In the category of internal groupoids over $B$, the torsion objects are Beck modules over $B$ \cite{Beck} and the torsion-free ones are again the equivalence relations on $B$.
In the category of crossed modules \cite{Whitehead} over a group $B$, the torsion objects are the $B$-modules and the torsion-free ones are the normal subgroups of $B$. The same descriptions hold if the category of groups is replaced by any semi-abelian category in the sense of \cite{JMT}. We conclude the paper by showing that protomodularity is not necessary for $(\mathcal{T},\mathcal{F})$ to be a torsion theory in the whole of $\mathcal{C}$: indeed, this is the case for the slice category $(\mathsf{Mon}\downarrow G)$ of monoids over a group $G$ all of whose elements have finite order, where $\mathcal{T}$ is the category of monoids and $\mathcal{F}$ is the poset of subgroups of $G$.

\section{The canonical near-torsion theory}

\begin{prop}\label{quasi-pointed}
	For a category $\mathcal{C}$ with an initial object $0$ and a terminal object $1$, the following conditions are equivalent:
	\begin{itemize}
		\item [(a)] the morphism $0\to1$ is a monomorphism;
		\item [(b)] for every object $A$ in $\mathcal{C}$, the uniquely determined morphism $!_A:0\to A$ is a monomorphism;
		\item [(c)] for every object $A$ in $\mathcal{C}$, there exists at most one morphism $A\to0$;
		\item [(d)] the forgetful functor $(\mathcal{C}\downarrow0)\to\mathcal{C}$ is fully faithful.
	\end{itemize}
\end{prop}
\begin{proof}
	This is very easy: For (a)$\Rightarrow$(b), just use the composite $0\to A\to 1$. (b)$\Rightarrow$(a) is trivial. For (a)$\Leftrightarrow$(c) use the diagram $$\xymatrix{A\ar@<0.5ex>[r]\ar@<-0.5ex>[r]&0\ar[r]&1}$$ (d) holds if and only if every diagram of the form $$\xymatrix{A\ar[dr]\ar[rr]&&B\ar[dl]\\&0}$$ commutes, which gives (d)$\Leftrightarrow$(c).
\end{proof}
\begin{defi}
	\emph{A category $\mathcal{C}$ with an initial object $0$ and a terminal object $1$ is said to be} quasi-pointed \emph{if it satisfies the equivalent conditions of Proposition \ref{quasi-pointed}.}
\end{defi}
The term ``quasi-pointed'' was first used by D. Bourn, who defined it in \cite{[B2001]} by the condition we labeled here as \ref{quasi-pointed}(a).

An object $A$ in a category with a terminal object $1$, is often called \textit{subterminal} if the uniquely determined morphism $A\to 1$ is a monomorphism. On the other hand, we will call an object $A$ in a quasi-pointed category $0$\textit{-supported} if it admits a morphism $A\to0$. In particular, $A$ is 0-supported and subterminal at the same time if and only if it is initial; we might then informally write $A=0$.\vspace{1mm}

\textit{From now on we will assume that our ambient category $\mathcal{C}$ is quasi-pointed and write $\mathcal{T}$ and $\mathcal{F}$ for the full subcategories of $\mathcal{C}$ with objects all $0$-supported objects in $\mathcal{C}$ and all subterminal objects in $\mathcal{C}$, respectively. }

\begin{lem}\label{torsion}
	The following conditions on an object $A$ in $\mathcal{C}$ are equivalent:
	\begin{itemize}
		\item [(a)] $A\in\mathcal{T}$;
		\item [(b)] there is a (unique) morphism $A\to0$, which is a universal arrow for the inclusion functor $\mathcal{F}\to\mathcal{C}$, that is, for every morphism $\alpha:A\to F$ with $F\in\mathcal{F}$, the diagram $$\xymatrix{A\ar[dr]_\alpha\ar[r]&0\ar[d]\\&F}$$ commutes. 
	\end{itemize}
\end{lem}
\begin{proof}
	(a)$\Rightarrow$(b): Just note that we have a monomorphism $F\to1$, and every two parallel morphisms with codomain $1$ are equal. (b)$\Rightarrow$(a) is trivial.  
\end{proof}
\begin{lem}\label{torsion-free}
	For every object $A$ in $\mathcal{F}$, the (unique) morphism $0\to A$ is a couniversal arrow for the inclusion functor $\mathcal{T}\to\mathcal{C}$, that is, for every morphism $\alpha:T\to A$ with $T\in\mathcal{T}$, the diagram $$\xymatrix{T\ar[d]\ar[dr]^\alpha\\0\ar[r]&A}$$ commutes.
\end{lem}
\begin{proof}
	Now we have a monomorphism $A\to1$, and again we can use the fact that every two parallel morphisms with codomain $1$ are equal.
\end{proof}
We then recall the following definition from \cite{[GJ2026]}:
\begin{defi}\label{near}
	\emph{A pair $(\mathcal{T},\mathcal{F})$ of full replete subcategories of a category $\mathcal{C}$ will be called a} near-torsion theory \emph{in $\mathcal{C}$ if it satisfies the following conditions:}
	\begin{itemize}
		\item [(a)] \emph{$\mathcal{T}$ is a coreflective subcategory of $\mathcal{C}$, that is, the inclusion functor $\mathcal{T}\to\mathcal{C}$ has a left inverse right adjoint $t:\mathcal{C}\to\mathcal{T}$;}
		\item [(b)] \emph{$\mathcal{F}$ is a reflective subcategory of $\mathcal{C}$, that is, the inclusion functor $\mathcal{F}\to\mathcal{C}$ has a left inverse left adjoint $f:\mathcal{C}\to\mathcal{F}$;}
		\item [(c)] \emph{$A\in\mathcal{F}\Rightarrow t(A)\in\mathcal{F}$;}
		\item [(d)] \emph{$A\in\mathcal{T}\Rightarrow f(A)\in\mathcal{T}$.}
	\end{itemize}
\end{defi}
From the two previous lemmas we obtain:
\begin{cor}
	The following conditions are equivalent:
	\begin{itemize}
		\item [(a)] $\mathcal{T}$ and $\mathcal{F}$ are coreflective and reflective, respectively, in $\mathcal{C}$;
		\item [(b)] $(\mathcal{T},\mathcal{F})$ is a near-torsion theory in $\mathcal{C}$ in the sense of Definition \ref{near} .\qed
	\end{itemize}
\end{cor}
\begin{lem}\label{couniversal}
	The following conditions on a morphism $m:T\to A$ in $\mathcal{C}$ with $T\in\mathcal{T}$ are equivalent:
	\begin{itemize}
		\item [(a)] $m$ is a couniversal arrow for the inclusion functor $\mathcal{T}\to\mathcal{C}$;
		\item [(b)] the diagram $$\xymatrix{A&T\ar[l]_-m\ar[r]&0}$$ is a product diagram.
	\end{itemize}
	Furthermore, if these conditions hold, then $m$ is a monomorphism.
\end{lem}
\begin{proof}
	(a)$\Rightarrow$(b): Suppose (a) holds and consider the diagram $$\xymatrix{&X\ar[dl]_\varphi\ar@{.>}[d]^\psi\ar[dr]\\A&T\ar[l]^-m\ar[r]&0}$$ where the solid arrows are given. This forces $X$ to belong to $\mathcal{T}$ and so determines a unique $\psi$ making the left-hand triangle commute. But then the right-hand triangle also commutes by \ref{quasi-pointed}(c). (b)$\Rightarrow$(a): Use the same diagram and \ref{quasi-pointed}(c) again. The fact that $m$ is a monomorphism follows from the fact that the two product projections are always jointly monic and \ref{quasi-pointed}(c). 
\end{proof}
\begin{cor}\label{coreflection}
	$\mathcal{T}$ is a coreflective subcategory of $\mathcal{C}$ if and only if, for every object $A$ in $\mathcal{C}$, the product $A\times0$ exists. Furthermore, in this case it is mono-coreflective.\qed
\end{cor}
Note that we could easily deduce Lemma \ref{torsion-free} from Lemma \ref{couniversal}.  Furthermore, under the assumption that any morphism $A \rightarrow 1$ admits a (strong epi, mono)-factorization, Lemma \ref{torsion} can be deduced from the following one, since in a quasi-pointed category, the unique morphism $A \rightarrow 0$ is a split epimorphism for any $A \in \mathcal T$:
\begin{lem}\label{Lemma-factorisation}
	Suppose every morphism in $\mathcal{C}$ of the form $A\to1$ admits a (strong epi, mono)-factorization. Then the following conditions on a morphism $e:A\to F$ in $\mathcal{C}$ with $F\in\mathcal{F}$ are equivalent:
	\begin{itemize}
		\item [(a)] $e$ is a universal arrow for the inclusion functor $\mathcal{F}\to\mathcal{C}$;
		\item [(b)] $e$ is a strong epimorphism.
	\end{itemize} 
\end{lem}
\begin{proof}
	 For (b)$\Rightarrow$(a), use the orthogonality diagram $$\xymatrix{A\ar[d]_\varphi\ar[r]^-e&F\ar@{.>}[dl]\ar[d]\\X\ar[r]&1}$$ where $\varphi$ is any morphism from $A$ to a subterminal object. For (a)$\Rightarrow$(b), factor the morphism $A\to1$ as a strong epimorphism $p:A\to I$ followed by a monomorphism $I\to1$; then $I\in\mathcal{F}$ and, by (b)$\Rightarrow$(a), $p$ is a universal arrow for the inclusion functor $\mathcal{F}\to\mathcal{C}$, so that $e$ is isomorphic to the strong epimorphism $p$.
\end{proof}
\begin{cor}\label{reflection}
	$\mathcal{F}$ is a reflective subcategory of $\mathcal{C}$ whenever every morphism in $\mathcal{C}$ of the form $A\to1$ admits a (strong epi, mono)-factorization. Furthermore, in this case it is strong epi-reflective.\qed
\end{cor}

Summarizing and using the notation of \cite{[GJ2026]} we obtain:
\begin{teo}\label{counit-projection}
	Let $\mathcal{C}$ be a quasi-pointed category in which, for every $A$, the product $A\times0$ exists and the morphism $A\to1$ admits a (strong epi, mono)-factorization. Then $$(\mathcal{T},\mathcal{F})=(\mathcal{C},\mathcal{T},\mathcal{F},t,\varepsilon,f,\eta)$$ is a near-torsion theory in $\mathcal{C}$, in which, for every object $A$ in $\mathcal{C}$:
	\begin{itemize}
		\item [(a)] $\varepsilon_A:t(A)\to A$ can be chosen to be the product projection $A\times0\to A$, and it is a monomorphism;
		\item [(b)] $\eta_A:A\to f(A)$ can be chosen to be any strong epimorphism $A\to F$ with subterminal $F$.
	\end{itemize}
\end{teo}

\section{From near-torsion to torsion}

In this section we keep the notation and assumptions of Theorem \ref{counit-projection}.
\begin{lem}\label{lemma-pullback}
	For every object $A$ in $\mathcal{C}$, the diagram \eqref{pushout} is a pullback.
\end{lem}
\begin{proof}
	This follows from Lemma \ref{couniversal} and the fact that $f(A)\to1$ is a monomorphism, since the diagram \begin{equation}\label{pullback2}
	\xymatrix{t(A)\ar[d]_{\varepsilon_A}\ar[r]&0\ar[d]\ar[ddr]\\A\ar[drr]\ar[r]^-{\eta_A}&f(A)\ar[dr]\\&&1}\end{equation} commutes.
\end{proof}
Recall from \cite{[GJ2026]} that, given a near-torsion theory $(\mathcal{T},\mathcal{F})=(\mathcal{C},\mathcal{T},\mathcal{F},t,\varepsilon,f,\eta)$, 
one can associate the subcategories $\mathring{\mathcal{T}}$, $\mathring{\mathcal{F}}$, and $\mathcal{C}_{\mathcal{T},\mathcal{F}}$ of $\mathcal{C}$: they are defined as follows:
\begin{itemize}
\item $T\in\mathring{\mathcal{T}}$ if and only if $T\in\mathcal{T}$ and $\eta_T$ is an epimorphism;
\item $F\in\mathring{\mathcal{F}}$ if and only if $F\in\mathcal{F}$ and $\varepsilon_F$ is a monomorphism;
\item $B\in\mathcal{C}_{\mathcal{T},\mathcal{F}}$ if and only if $t(B)\in\mathring{\mathcal{T}}$, $f(B)\in\mathring{\mathcal{F}}$, the diagram $$\xymatrix{t(B)\ar[d]_{\varepsilon_B}\ar[r]^-{t(\eta_B)}&tf(B)\ar[d]^{\varepsilon_{f(B)}}\\B\ar[r]_-{\eta_B}&f(B)}$$ is a pullback, and the diagram $$\xymatrix{t(B)\ar[d]_{\varepsilon_B}\ar[r]^-{\eta_{t(B)}}&ft(B)\ar[d]^{f(\varepsilon_B)}\\B\ar[r]_-{\eta_B}&f(B)}$$ is a pushout.
\end{itemize}
The pair $(\mathring{\mathcal{T}},\mathring{\mathcal{F}})$ is then a (non-pointed) torsion theory in $\mathcal{C}_{\mathcal{T},\mathcal{F}}$. Moreover, $\mathring{\mathcal{T}}=\mathcal{T}\cap\mathcal{C}_{\mathcal{T},\mathcal{F}}$, $\mathring{\mathcal{F}}=\mathcal{F}\cap\mathcal{C}_{\mathcal{T},\mathcal{F}}$, and $\mathcal{C}_{\mathcal{T},\mathcal{F}}$ is the largest full subcategory $\mathcal{C}'$ of $\mathcal{C}$ such that every object of $\mathcal{C}'$ satisfies the conditions defining a (non-pointed) torsion theory in \cite{[GJ2026]}, with the monomorphisms, epimorphisms, pullbacks and pushouts involved being computed in $\mathcal{C}$.

In the quasi-pointed context the two diagrams above coincide: since $f(B)\in\mathcal{F}$ we have $tf(B)=0$ by Lemma \ref{torsion-free}, and since $t(B)\in\mathcal{T}$ we have $ft(B)=0$ by Lemma \ref{torsion}, so that both of them become the diagram \eqref{pushout}. Moreover $\varepsilon_{f(B)}\colon 0\to f(B)$ is a monomorphism by \ref{quasi-pointed}(b), while $\eta_{t(B)}\colon t(B)\to 0$ is a split epimorphism. The standing assumptions of the two lemmas of \cite{[GJ2026]} characterizing $\mathcal{C}_{\mathcal{T},\mathcal{F}}$ are therefore satisfied by every object of $\mathcal{C}$, and \eqref{pushout} is always a pullback by Lemma \ref{lemma-pullback}: the only condition left is that it should also be a pushout. We then have the following 
\begin{teo}\label{thm:3.2}
	Under the assumptions of Theorem \ref{counit-projection}, we have:
	\begin{itemize}
		\item [(a)] $\mathring{\mathcal{T}}=\mathcal{T}$;
		\item [(b)] $\mathring{\mathcal{F}}=\mathcal{F}$;
		\item [(c)] an object $A$ in $\mathcal{C}$ belongs to $\mathcal{C}_{\mathcal{T},\mathcal{F}}$ if and only if the diagram $$\xymatrix{t(A)\ar[d]_{\varepsilon_A}\ar[r]&0\ar[d]\\A\ar[r]_-{\eta_A}&f(A)}$$ is a pushout.\qed
	\end{itemize}
\end{teo}

\begin{remark}
	\emph{In terminology of \cite{[B2001]}, condition \ref{thm:3.2}(c), together with the assertion of Lemma \ref{lemma-pullback}, means that \begin{equation}\label{SES}
	\xymatrix{&t(A)\ar[r]^-{\varepsilon_A}&A\ar[r]^-{\eta_A}&f(A) &}\end{equation}
	is a short exact sequence. This also agrees with the terminology of \cite{[FF2020]} and \cite{[FFG2021]}, since in these papers $\mathcal{T}\cap\mathcal{F}=\mathcal Z$ is the subcategory of ``trivial objects'', which in our quasi-pointed context is the full subcategory of initial objects. Kernels, cokernels and the exactness of the short exact sequence \eqref{SES} are then defined relatively to the ideal of morphisms induced by $\mathcal Z$. Note that the pair $(\mathcal T , \mathcal F)$ is then a pretorsion theory in $\mathcal{C}_{\mathcal{T},\mathcal{F}}$ in the sense of \cite{[FF2020],[FFG2021]} and it is one in $\mathcal C$  itself exactly when every object satisfies \ref{thm:3.2}(c).
		}\end{remark}
Consider examples, starting with two `trivial' ones:
\begin{example}
	If $\mathcal{C}$ is pointed and has an initial (= terminal) object, then it is automatically quasi-pointed and satisfies all our conditions. In this case $\mathcal{T}=\mathcal{C}=\mathcal{C}_{\mathcal{T},\mathcal{F}}$ and every subterminal object in $\mathcal{C}$ is terminal. Briefly, $(\mathcal{T},\mathcal{F})=(\mathcal{C},0)$ is a torsion theory in $\mathcal{C}$.
\end{example}
\begin{example}
	Suppose $\mathcal{C}$ has a strict initial object, that is, every morphism of the form $A\to0$ in it is an isomorphism. Then $\mathcal{C}$ is automatically quasi-pointed again, and we can write $A\times0=0$ for any object $A$. Then every $0$-supported object in $\mathcal{C}$ is initial, and  $\mathcal{T}$ is the coreflective full subcategory of initial objects. If, moreover, every morphism of the form $A\to1$ admits a (strong epi, mono)-factorization, so that the assumptions of Theorem \ref{counit-projection} are satisfied, then the diagram in \ref{thm:3.2}(c) is a pushout if and only if $\eta_A\colon A\to f(A)$ is an isomorphism, that is, if and only if $A$ is subterminal. In this case we therefore obtain nothing but $(\mathcal{T},\mathcal{F})=(0,\mathcal{F})$ as a torsion theory in $\mathcal{F}=\mathcal{C}_{\mathcal{T},\mathcal{F}}$.
\end{example}

\begin{examples}\label{Lie+crossed}
As follows from Theorem~\ref{thm:3.2} and the results of \cite{[B2001]},
$\mathcal{C}_{\mathcal{T},\mathcal{F}}=\mathcal{C}$ whenever $\mathcal{C}$ is
\emph{sequentiable} in the sense of \cite{[B2001]}, that is, whenever
$\mathcal{C}$ is regular and protomodular \cite{[B1991]} and $0\to 1$ is a
monomorphism in $\mathcal{C}$. In other words, in a sequentiable category the
pair $(\mathcal{T},\mathcal{F})$ is always a (non-pointed) torsion theory. Since the term ``sequential'' is mentioned, let us also mention a special role of $\mathcal{T}$: every morphism of any exact sequence in $\mathcal{C}$, except possibly the last one, has its domain in $\mathcal T$. Let us examine some examples of such torsion theories $(\mathcal{T},\mathcal{F})$
in more detail.

{\bf Lie-Rinehart algebras.}  Let $R$ be a commutative unital ring and $A$ a commutative $R$-algebra.
We write $\mathsf{Der}(A)$ for the Lie $R$-algebra of $R$-derivations of $A$,
with its natural $A$-module structure. A \emph{Lie--Rinehart algebra}
\cite{Rinehart} over $(R,A)$ is a Lie $R$-algebra $L$ equipped with an
$A$-module structure and an \emph{anchor}
$\alpha \colon L \to \mathsf{Der}(A)$, which is both a morphism of Lie
$R$-algebras and a morphism of $A$-modules, subject to the Leibniz rule
\[
[x,\, a y] = a\,[x,y] + \alpha(x)(a)\, y
\qquad \text{for all } x, y \in L,\ a \in A.
\]
A morphism $(L,\alpha) \to (L',\alpha')$ is an $A$-linear morphism of Lie
$R$-algebras $f \colon L \to L'$ such that $\alpha' \circ f = \alpha$.
As observed in \cite{Couchman}, the category $\mathsf{LR}$ of Lie--Rinehart
algebras over $(R,A)$ is exact and protomodular. It is also quasi-pointed, since the morphism
from the initial object $(0,0)$ to the terminal object
$(\mathsf{Der}(A), 1_{\mathsf{Der}(A)})$ is a monomorphism, hence $\mathsf{LR}$  is a
sequentiable category.
Since a morphism $(L,\alpha)\to(0,0)$ exists if and only if $\alpha=0$, and
the unique morphism $(L,\alpha)\to(\mathsf{Der}(A), 1_{\mathsf{Der}(A)})$ is
$\alpha$ itself, the torsion subcategory of $0$-supported objects and the
torsion-free subcategory of subterminal objects are described, respectively, as follows:
\[
\mathcal{T} = \{(L,\alpha) \mid \alpha = 0\}
\qquad\text{and}\qquad
\mathcal{F} = \{(L,\alpha) \mid \alpha \text{ is injective}\}.
\]
In other words, $\mathcal{T}$ is the category of Lie--Rinehart algebras with
trivial anchor (for which the Leibniz rule reduces to the $A$-bilinearity of
the bracket), and $\mathcal{F}$ is equivalent to the poset of Lie--Rinehart
subalgebras of $\mathsf{Der}(A)$.

Given a Lie--Rinehart algebra $(L,\alpha)$, the canonical short exact
sequence \eqref{SES} associated with it is then represented by the diagram
\[
\begin{tikzcd}[row sep=large, column sep=large]
& \ker(\alpha) \arrow[r, tail] \arrow[rd, "0"'] & L \arrow[r, two heads] \arrow[d, "\alpha"] & \alpha(L) \arrow[ld, hook']  \\
& & \mathsf{Der}(A) & &
\end{tikzcd}
\]
where $\ker(\alpha)$ carries the zero anchor, and the image $\alpha(L)$ of
$\alpha$ has the inclusion into $\mathsf{Der}(A)$ as anchor.

{\bf Internal graphs, reflexive graphs and groupoids over a fixed group.}
Let $B$ be a group. We write $\mathsf{Gph}_B(\mathsf{Grp})$ for the category of internal graphs in the category
$\mathsf{Grp}$ of groups having $B$ as ``object of objects'', so that an object is a diagram
$$\begin{tikzcd}
X \arrow[r, shift left, "d"] \arrow[r, shift right, "c"'] & B
\end{tikzcd}$$
in the category $\mathsf{Grp}$, and morphisms in $\mathsf{Gph}_B(\mathsf{Grp})$ preserve the domain and codomain morphisms. 
Since such a graph is the same thing as a homomorphism $(d,c)\colon X\to B\times B$, the category $\mathsf{Gph}_B(\mathsf{Grp})$ is nothing but the slice category $(\mathsf{Grp}\downarrow B\times B)$: its initial object is $0\to B\times B$, its terminal object is $1_{B\times B} \colon B\times B \rightarrow B\times B$, and the morphism $0\to 1$ is the monomorphism $0\to B\times B$. The category $\mathsf{Gph}_B(\mathsf{Grp})$ is then quasi-pointed, exact and protomodular, hence in particular it is sequentiable. In this category an object $(X,d,c)$ is $0$-supported
if and only if $d$ and $c$ are trivial morphisms, and it is subterminal if and only
if it is an (internal) relation on $B$. Accordingly, in the category $\mathsf{Gph}_B(\mathsf{Grp})$ we have 
\[
\mathcal{T}\cong\mathsf{Grp}
\qquad\text{and}\qquad
\mathcal{F}\simeq \mathsf{Rel}_B(\mathsf{Grp}),
\]
the latter being the poset of internal relations on $B$ in $\mathsf{Grp}$.
For a graph
$(X,d,c)$ the canonical short exact sequence \eqref{SES} is then depicted as 
$$\begin{tikzcd}[row sep=large, column sep=large]
\ker(d)\cap\ker(c)
  \arrow[r, tail]
  \arrow[d, shift left, "0"]
  \arrow[d, shift right, "0"']
& X
  \arrow[r, two heads, "{(d,c)}"]
  \arrow[d, shift left, "c"]
  \arrow[d, shift right, "d"']
& (d,c)(X)
  \arrow[d, shift left, "\pi_2"]
  \arrow[d, shift right, "\pi_1"']
\\
B \arrow[r, equal] & B \arrow[r, equal] & B
\end{tikzcd}
$$
where $(d,c)(X) = \{(d(x),c(x))\mid x\in X\}$ is the direct image of $X$ under $(d,c) \colon X \rightarrow B \times B$, and $\ker(d)\cap\ker(c)$ is the kernel of $(d,c)$ in $\mathsf{Grp}$.
 
Next, let $\mathsf{RG}_B(\mathsf{Grp})$ be the category of \emph{reflexive} graphs in
$\mathsf{Grp}$ with $B$ as object of objects. An object $(X,d,c,e)$  in $\mathsf{RG}_B(\mathsf{Grp})$  is a graph
$(X,d,c)$ as above together with a  `section' homomorphism $e\colon B\to X$ satisfying
$d\circ e=c \circ e=1_B$, and morphisms in $\mathsf{RG}_B(\mathsf{Grp})$  are also required to preserve sections. The
initial object is the discrete reflexive graph $(B,1_B,1_B,1_B)$, the
terminal object is the indiscrete one $(B\times B,p_1,p_2,\Delta)$, and
the morphism $0\to 1$ is the diagonal $\Delta\colon B\to B\times B$, which is a monomorphism. The category $\mathsf{RG}_B(\mathsf{Grp})$ is then quasi-pointed, and it is also exact and protomodular (see \cite{[EG2006]}, for instance).
A morphism from $(X,d,c,e)$ to the
discrete reflexive graph is a homomorphism $f\colon X\to B$ with $f=d=c$, so
that $(X,d,c,e)$ is $0$-supported if and only if $d=c$, which means that it is nothing but the ``point'' $(X,d,e)$ over $B$ in the sense of \cite{[B1991]}. It is subterminal if and only if
$(d,c)$ is injective, that is, if and only if $X$ is a reflexive relation
on $B$, hence if and only if it is an equivalence relation on $B$, since $\mathsf{Grp}$ is a Mal'tsev category. For a reflexive graph $(X,d,c,e)$ the canonical short exact sequence \eqref{SES} is then depicted as 
\begin{equation}\label{SES-RG}
\begin{tikzcd}[row sep=large, column sep=large]
E
  \arrow[r, tail]
  \arrow[d, shift left=2, "c_E"]
  \arrow[d, shift right=2, "d_E"']
& X
  \arrow[r, two heads, "{(d,c)}"]
  \arrow[d, shift left=2, "c"]
  \arrow[d, shift right=2, "d"']
& R
  \arrow[d, shift left=2, "\pi_2"]
  \arrow[d, shift right=2, "\pi_1"']
\\
B \arrow[r, equal] \arrow[u, "e_E" description]
& B \arrow[r, equal] \arrow[u, "e" description]
& B \arrow[u, "\Delta" description]
\end{tikzcd}
\end{equation}
where $E=\{x\in X\mid d(x)=c(x)\}$ is the reflexive subgraph of ``loops'' of $X$, with $d_E =c_E$ the restriction of $d$ (or $c$) and $e_E$ the corestriction of $e$, and $R = (d,c)(X)$ is the equivalence relation on $B$ whose classes are the connected components of $X$.
Accordingly,
\[
\mathcal{T}\cong\mathsf{Pt}_B(\mathsf{Grp})
\qquad\text{and}\qquad
\mathcal{F}\simeq{\mathsf{Eq}_B(\mathsf{Grp})},
\]
where $\mathsf{Pt}_B(\mathsf{Grp})$ is the category of points over $B$,
which is equivalent to the category of $B$-groups, and $\mathsf{Eq}_B(\mathsf{Grp})$
is the poset of equivalence relations on $B$ in $\mathsf{Grp}$.

Let $\mathsf{Grpd}_B(\mathsf{Grp})$ be the full subcategory of $\mathsf{RG}_B(\mathsf{Grp})$
whose objects are the internal groupoids in $\mathsf{Grp}$ with $B$ as
object of objects. Recall that a reflexive graph $(X,d,c,e)$ in
$\mathsf{Grp}$ carries at most one groupoid structure $m \colon X \times_BX \rightarrow X$, and it carries one
if and only if $[\ker(d),\ker(c)]=1$ \cite{Lod82}.
The category $\mathsf{Grpd}_B(\mathsf{Grp})$ is quasi-pointed,
exact and protomodular, hence sequentiable \cite{[B2001]}.
Subterminal objects are again the equivalence relations on $B$, while the $0$-supported objects
are the groupoids with $d=c$, i.e. the points $(X,d,e)$ over $B$ with
$[\ker(d),\ker(d)]=1$. These are precisely the abelian group objects in
$\mathsf{Pt}_B(\mathsf{Grp})$, that is, the \emph{Beck modules} over $B$
\cite{Beck}. Accordingly,
\[
\mathcal{T}\simeq\mathsf{Ab}(\mathsf{Pt}_B(\mathsf{Grp}))
\qquad\text{and}\qquad
\mathcal{F}\simeq{\mathsf{Eq}_B(\mathsf{Grp})}.
\]
The canonical short exact sequence of a groupoid $(X,d,c, e,m)$ is
still depicted as in \eqref{SES-RG}, where now $E$ is the totally
disconnected groupoid of automorphisms of $X$, and $R$ is the equivalence relation on $B$ relating two objects of $X$ precisely when they are isomorphic in $X$, so that its classes are the connected components of $X$.

{\bf Precrossed and crossed modules over a fixed group.}
The two previous examples can be reformulated in terms of (pre)crossed modules. Recall that a \emph{precrossed module} over $B$ is a group homomorphism $\alpha\colon A\to B$ (the boundary map), together with an action of $B$ on $A$, written $(b,a)\mapsto {}^{b}a$, such that $\alpha({}^{b}a)=b\,\alpha(a)\,b^{-1}$ for all $b\in B$ and $a\in A$; it is a \emph{crossed module} \cite{Whitehead} when it also satisfies the 
identity ${}^{\alpha(a)}a'=a\,a'\,a^{-1}$ for all $a,a'\in A$. Let $\mathsf{PXMod}_B(\mathsf{Grp})$ and $\mathsf{XMod}_B(\mathsf{Grp})$ be the corresponding categories, whose morphisms are the $B$-equivariant homomorphisms commuting with the boundary maps. There are then category equivalences
\[
\mathsf{RG}_B(\mathsf{Grp})\simeq \mathsf{PXMod}_B(\mathsf{Grp})
\qquad\text{and}\qquad
\mathsf{Grpd}_B(\mathsf{Grp})\simeq \mathsf{XMod}_B(\mathsf{Grp})
\]
(see \cite{BS1976}, and also \cite{Lod82, [EG2006]}, for instance, for the case of precrossed modules) sending a reflexive graph $(X,d,c,e)$ to $(\ker(d),\alpha)$, where $\alpha$ is the restriction of $c$ to $\ker(d)$ and $B$ acts on $\ker(d)$ by conjugation through $e$. Under these equivalences a precrossed module $\alpha\colon A\to B$ is $0$-supported if and only if $\alpha$ is trivial, and it is subterminal if and only if $\alpha$ is injective. The subgraph of loops $E$ and the equivalence relation $R$ of a reflexive graph correspond, respectively, to $\ker(\alpha)$, with trivial boundary and the restricted action, and to the normal subgroup $\alpha(A)$ of $B$, with the inclusion as boundary and the conjugation action. In $\mathsf{PXMod}_B(\mathsf{Grp})$ the torsion subcategory is therefore the category of $B$-groups, the torsion-free subcategory is the poset $\mathsf{N}(B)$ of normal subgroups of $B$, which is of course equivalent to $\mathsf{Eq}_B(\mathsf{Grp})$, and the canonical short exact sequence \eqref{SES} of a precrossed module $\alpha\colon A\to B$ is represented by the diagram
\[
\begin{tikzcd}[row sep=large, column sep=large]
\ker(\alpha) \arrow[r, tail] \arrow[rd, "0"'] & A \arrow[r, two heads] \arrow[d, "\alpha"] & \alpha(A) \arrow[ld, hook']  \\
& B &
\end{tikzcd}
\]
In $\mathsf{XMod}_B(\mathsf{Grp})$ the same description holds, with the only difference that the identity ${}^{\alpha(a)}a'=a\,a'\,a^{-1}$ forces a crossed module with trivial boundary to be abelian, so that the torsion subcategory is the category $\mathsf{Mod}_{\mathbb{Z}[B]}$ of $B$-modules. The diagram above then recovers the classical decomposition of a crossed module $\alpha\colon A\to B$ into its kernel, which is a $B$-module and a central subgroup of $A$, and its image, which is a normal subgroup of $B$.

More generally, everything we have said holds for internal graphs, reflexive graphs and groupoids with a fixed object of objects $B$ in any semi-abelian category $\mathcal{A}$ in the sense of \cite{JMT}, and hence, thanks to the results in \cite{Janelidze}, for internal precrossed and crossed modules over $B$ in $\mathcal{A}$. The torsion objects of $\mathsf{XMod}_B(\mathcal A)$ are the \emph{internal $B$-modules}, that is, the abelian objects of $\mathsf{Pt}_B(\mathcal{A})$, while the torsion-free objects correspond to the \emph{normal subobjects of }$B$, which classify the internal equivalence relations on $B$ since $\mathcal{A}$ is semi-abelian.

\end{examples}

\begin{example}
	Let us take $\mathcal{C}$ to be the quasi-pointed category $(\mathsf{Mon}\downarrow G)$, where $\mathsf{Mon}$ is the category of monoids, and $G$ is a finite group.  
	
	Here $\mathcal{T}$ is isomorphic to $\mathsf{Mon}$, identified with the objects $(M,\varphi )$ whose $\varphi \colon M \rightarrow G$ is trivial, while $\mathcal{F}$ is equivalent to the poset of submonoids of $G$, that is, of subgroups of $G$, since $G$ is finite.
The assumptions of Theorem \ref{counit-projection} are easily seen to be satisfied, with $t(M, \varphi)= \varphi^{-1}(1)$ and $f(M,\varphi) = \varphi(M)$.  Since pushouts in $\mathcal{C}$ are computed as in $\mathsf{Mon}$, the diagram in Theorem \ref{thm:3.2}(c) is a pushout if and only if the canonical surjective morphism $\pi \colon  M/{\sim} \to \varphi(M)$ is an isomorphism, where $\sim$ is the congruence on $M$ generated by the pairs $(n,1)$ with $n \in \varphi^{-1}(1)$. For every $m \in M$ the element $\varphi(m)$ of the finite group $G$ has a finite order $k$, so that $m^{k} \in \varphi^{-1}(1)$ and $[m]^{k} = [1]$ in $ M/{\sim} $; being generated by invertible elements, the monoid $M/{\sim}$ is a group. If $\pi([m]) = 1$, then $m \in \varphi^{-1}(1)$, and so $[m] = [1]$; hence $\pi$ is injective, and therefore a group isomorphism. This gives $\mathcal{C}_{\mathcal{T},\mathcal{F}} = \mathcal{C}$, and $(\mathcal{T},\mathcal{F})$ is a torsion theory in $\mathcal{C}$, even though $\mathcal{C}$ is not protomodular.
	Note that the finiteness of $G$ has only been used to guarantee that every element of $G$ has a finite order: the same conclusion therefore holds whenever $G$ is a torsion group.
\end{example}

{\bf Acknowledgement} The first author's research was supported by the Fonds de la Recherche Scientifique -- FNRS under Grant CDR No.~J.0092.26. The second author gratefully acknowledges the financial support of UCLouvain, which made possible a research visit in July 2026 during which part of this work was carried out.
{}
	

\begin{thebibliography}{99}
		
		
		
		\bibitem{Beck} J. Beck, Triples, algebras and cohomology, Ph.D. thesis, Columbia University, 1967; reprinted in Repr. Theory Appl. Categ. \textbf{2} (2003), 1--59.
		
		\bibitem{[B1991]} D. Bourn, Normalization equivalence, kernel equivalence and affine categories, in: Category Theory (Como, 1990), Lecture Notes in Math. \textbf{1488}, Springer, Berlin, 1991, 43-62
		
		\bibitem {[B2001]} D. Bourn, $3\times3$-Lemma and Protomodularity, J. Algebra \textbf{236} (2001), 778-795
		
	\bibitem{BS1976} R. Brown and C. B. Spencer, $\mathcal{G}$-groupoids, crossed modules and the fundamental groupoid of a topological group, Indag. Math. \textbf{38} (1976), 296--302.
		
		\bibitem{Couchman} J.~Couchman, Crossed modules and internal categories of Lie--Rinehart algebras, Master of Research thesis, Macquarie University, 2017.
		
		\bibitem {[EG2006]} T. Everaert and M. Gran, Precrossed modules and Galois theory, J. Algebra \textbf{297} (2006), 292-309.
		
		
		\bibitem {[FF2020]} A. Facchini and C. A. Finocchiaro, Pretorsion theories, stable category and preordered sets, Ann. Mat. Pura Appl. \textbf{199} (2020), 1073-1089
		
		\bibitem {[FFG2021]} A. Facchini, C. A. Finocchiaro, and M. Gran, Pretorsion theories in general categories, J. Pure Appl. Algebra \textbf{225} (2021), 106503
		
		\bibitem {[GJ2026]} M. Gran and G. Janelidze, Torsion and exactness, preprint (2026)  arXiv:2608.23804
		
		\bibitem{Janelidze} G.~Janelidze, Internal crossed modules, Georgian Math.\ J.\ \textbf{10} (2003), 99--114.
		
	
\bibitem{JMT} G. Janelidze, L. M\'arki and W. Tholen, Semi-abelian categories, J. Pure Appl. Algebra \textbf{168} (2002), 367--386.
		
		\bibitem{Lod82} J.-L. Loday, \emph{Spaces with finitely many non-trivial homotopy groups}, J. Pure Appl. Algebra \textbf{24} (1982), no.~2, 179--202.
		
		\bibitem{Rinehart} G.~S.~Rinehart, Differential forms on general commutative algebras, Trans.\ Amer.\ Math.\ Soc.\ \textbf{108} (1963), 195--222.
		
	\bibitem{Whitehead} J.~H.~C.~Whitehead, Combinatorial homotopy. II,  Bull.\ Amer.\ Math.\ Soc.\ \textbf{55} (1949), 453--496.
		
	\end{thebibliography}
\end{document}